\documentclass[12pt]{amsart}
\usepackage{pinlabel}
\usepackage{amsfonts,amsmath,amscd,xcolor,graphicx,amssymb}
\usepackage{caption}
\usepackage{epstopdf}
\usepackage{a4wide}

\theoremstyle{plain}
\newtheorem{theorem}{Theorem}[section]

\newtheorem{lemma}[theorem]{Lemma}
\newtheorem{corollary}[theorem]{Corollary}

\newtheorem{question}[theorem]{Question}

\theoremstyle{definition}

\newcommand{\N}{\operatorname{N}}

\theoremstyle{remark}

\makeatletter
\@namedef{subjclassname@2020}{\textup{2020} Mathematics Subject Classification}
\makeatother

\begin{document}

\title[The Powell Conjecture]{The Powell Conjecture for the genus-three Heegaard splitting of the $3$-sphere}

\author[S. Cho]{Sangbum Cho}
\thanks{The first-named author is supported by the National Research Foundation of Korea(NRF) grant funded by the Korea government(MSIT) (NRF-2021R1F1A1060603 and RS-2024-00456645).}
\address{Department of Mathematics Education, Hanyang University, Seoul 04763, Korea, and School of Mathematics, Korea Institute for Advanced Study, Seoul 02455, Korea}
\email{scho@hanyang.ac.kr}

\author[Y. Koda]{Yuya Koda}
\thanks{The second-named author is supported by JSPS KAKENHI Grant Numbers JP20K03588, JP21H00978, and JP23H05437.}
\address{Department of Mathematics, Hiyoshi Campus, Keio University, Yokohama 223-8521, Japan, and
International Institute for Sustainability with Knotted Chiral Meta Matter (WPI-SKCM$^2$), Hiroshima University, Higashi-Hiroshima 739-8526, Japan}
\email{koda@keio.jp}

\author[J. H. Lee]{Jung Hoon Lee}
\thanks{The third-named author is supported by the National Research Foundation of Korea(NRF) grant funded by the Korea government(MSIT) (RS-2023-00275419), and supported by Global - Learning \& Academic research institution for Master’s\,$\cdot$\,PhD students, and Postdocs(LAMP) Program of the National Research Foundation of Korea(NRF) grant funded by the Ministry of Education (No. RS-2024-00443714).}
\address{Department of Mathematics and Institute of Pure and Applied Mathematics, Jeonbuk National University, Jeonju 54896, Korea}
\email{junghoon@jbnu.ac.kr}

\date{\today}

\begin{abstract}
The Powell Conjecture states that the Goeritz group of the Heegaard splitting of the $3$-sphere is finitely generated; furthermore, four specific elements suffice to generate the group. Zupan demonstrated that the conjecture holds if and only if the reducing sphere complexes are all connected. In this work, we establish the connectivity of the reducing sphere complex for the genus-$3$ case, thereby confirming the Powell Conjecture in genus $3$. Additionally, we propose a potential framework for extending this approach to Heegaard splittings of higher genera.
\end{abstract}

\maketitle

\begin{small}
\hspace{2em}

\hspace{2em}

\hspace{7.8em} reducing sphere, weak reducing disk
\end{small}

\section{Introduction}\label{sec:introduction}

Let $(V, W; \Sigma)$ be a \textit{Heegaard splitting} of a closed orientable $3$-manifold $M$.
Specifically, $\Sigma$ is a closed orientable surface in $M$ that decomposes $M$ into two handlebodies $V$ and $W$.
The surface $\Sigma$ is called the \textit{Heegaard surface}, and the genus of $\Sigma$ is referred to as the \textit{genus} of the splitting.
The \textit{Goeritz group} of the splitting is defined as the automorphism group of the splitting, that is, the group of isotopy classes of orientation-preserving self-homeomorphisms of $M$ that preserve $V$ and $W$ setwise.

Due to a result by Waldhausen \cite{Wal68}, the Heegaard splitting of the $3$-sphere $S^3$ is uniquely determined by its genus up to isotopy.
Thus, we denote the Goeritz group of the genus-$g$ Heegaard splitting of $S^3$ simply by $\mathcal{G}_g$.
It is straightforward to see that $\mathcal{G}_0 = 1$ and $\mathcal{G}_1 = \mathbb{Z} / 2 \mathbb{Z}$; however, the difficulty level increases dramatically when the genus exceeds $1$.
In 1933, Goeritz \cite{Goe33} provided a finite generating set for $\mathcal{G}_2$, which is one reason why the automorphism groups of Heegaard splittings are referred to as Goeritz groups today.
However, there was little progress in studying the Goeritz groups $\mathcal{G}_g$ for $g \geq 3$ during the half-century following Goeritz's work.

In 1980, Powell published a paper \cite{Pow80} claiming to have found a finite generating set for $\mathcal{G}_g$ when $g \geq 3$.
However, in 2004, Scharlemann \cite{Sch04} pointed out a critical gap in Powell's proof.
Despite this, the generating set proposed by Powell remains plausible, and the conjecture that this set generates $\mathcal{G}_g$ is now known as the \textit{Powell Conjecture}.
It is worth noting that Scharlemann \cite{Scharlemann} demonstrated that one of the five elements in Powell's generating set is redundant.
Hence, the Powell Conjecture asserts that the remaining four elements, as shown in Figure \ref{Fig1}, generate $\mathcal{G}_g$.
In the figure, the four elements are described as isotopies of $S^3$ from the identity to a homeomorphism that preserves each handlebody of the splitting setwise.

\begin{figure}[htbp]
\begin{center}
\labellist
\pinlabel {$\varphi_\omega$} [Bl] at 22 -3
\pinlabel {$\varphi_{\eta}$} [Bl] at 172 -3
\pinlabel {$\varphi_{\eta_{1,2}}$} [Bl] at 300 -3
\pinlabel {$\varphi_\theta$} [Bl] at 438 -3
\endlabellist
\includegraphics[width=15.5cm,clip]{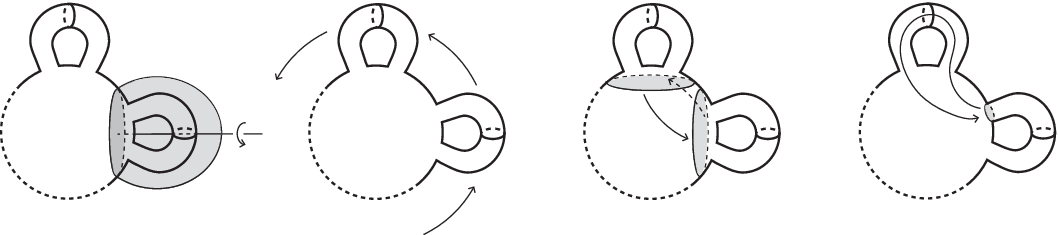}
\vspace{0.2cm}
\captionof{figure}{The four generators of $\mathcal{G}_g$ proposed in the Powell Conjecture.}
\label{Fig1}
\end{center}
\end{figure}
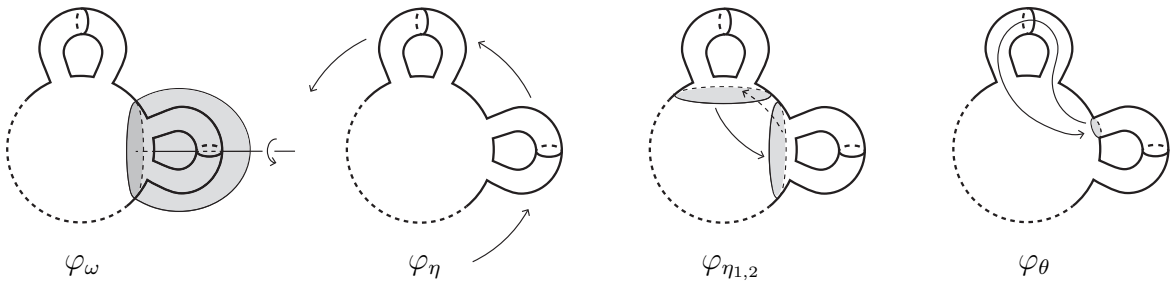

Since the pioneering work of Scharlemann \cite{Sch04}, the study of Goeritz groups has evolved significantly, drawing techniques from diverse areas such as singularity theory, dynamical systems, and geometric group theory.
For a comprehensive overview, we refer the reader to the introduction of \cite{CK19}.
Despite various advances in understanding Goeritz groups, the Powell Conjecture remains one of the central open problems in this field.

A key approach to tackling the conjecture is to identify a nice simplicial complex on which the Goeritz group acts naturally.
For instance, Scharlemann \cite{Sch04} provided a modern proof of Goeritz's classical result on $\mathcal{G}_2$ using the \textit{reducing sphere complex}.
This work was extended by Akbas, who derived a finite presentation of $\mathcal{G}_2$ in \cite{Akb}.
Shortly thereafter, the first-named author independently obtained the same presentation of $\mathcal{G}_2$ in \cite{Cho08} using a different complex, known as the \textit{primitive disk complex}.
However, many of the techniques developed for the genus-$2$ case have proven not be applied directly to the case of higher genera.
For example, in Heegaard splittings of genus greater than $2$, the standard disk surgery techniques, which are crucial for proving the connectivity (and contractibility) of the primitive disk complex for the genus-$2$ splitting, are no longer applicable (see \cite{CKL20}).

Recently, Freedman and Scharlemann \cite{Freedman-Scharlemann} presented a proof of the Powell Conjecture for $\mathcal{G}_3$ by employing a technique from singularity theory called \textit{graphics}.
On the other hand, Zupan \cite{Zupan} showed that if the reducing sphere complexes are connected for all relevant genera, then the Powell Conjecture holds.

Here, we review the definition of the reducing sphere complex, as it is a central object of study in this paper.
Let $(V, W; \Sigma)$ be a genus-$g$ Heegaard splitting of $S^3$ with $g \geq 3$.
A $2$-sphere $P$ in $S^3$ is called a \textit{reducing sphere} for $\Sigma$ if $P \cap \Sigma$ is a single circle that is essential in $\Sigma$.
If $P$ is a reducing sphere for $\Sigma$, then $P \cap V$ and $P \cap W$ are essential separating disks in $V$ and $W$, respectively.
The \textit{reducing sphere complex} $\mathcal{R}(\Sigma)$ is a simplicial complex defined as follows:
the vertices of $\mathcal{R}(\Sigma)$ are the isotopy classes of reducing spheres for $\Sigma$, and $k+1$ distinct vertices span a $k$-simplex if there exist representative spheres of the vertices that are pairwise disjoint.

Zupan's result \cite{Zupan} precisely states that for any integer $g \geq 3$, the Powell Conjecture for $\mathcal{G}_k$ holds for all $3 \leq k \leq g$ if and only if the reducing sphere complex $\mathcal{R}(\Sigma_k)$ is connected for all $3 \leq k \leq g$, where $\Sigma_k$ is the genus-$k$ Heegaard surface of $S^3$.

\bigskip

In this paper, we provide a short proof of the connectivity of the reducing sphere complex for the genus-$3$ Heegaard splitting of $S^3$, which immediately gives an alternative proof of the Powell Conjecture for $\mathcal{G}_3$.
Our argument relies on results from the theory of \textit{topologically minimal surfaces}, developed by Bachman \cite{Bachman1, Bachman2, Bachman3}.
The essential property we use is that the genus-$3$ Heegaard surface for $S^3$ is not \textit{critical}, meaning that the disk complex for the genus-$3$ Heegaard surface is simply connected.
Indeed, it has been claimed in \cite{Appel} and in \cite{Campisi-Torres} that the genus-$g$ Heegaard surface $\Sigma$ in $S^3$ with $g \geq 2$ is topologically minimal with index $2g - 1$.
This implies that $\pi_i(\mathcal{D}(\Sigma))$ is trivial for $i \leq 2g - 3$, leading directly to the conclusion that $\mathcal{D}(\Sigma)$ is simply connected.
While the arguments in \cite{Campisi-Torres} raise some questions regarding their completeness, and \cite{Appel} has not been published in a refereed journal, our proof does not rely on these works in their entirety.
Instead, for our purposes, we only require the simple connectivity of $\mathcal{D}(\Sigma)$.
To address this, we provide an independent proof based primarily on Theorem 3.2.B in \cite{Iva02} and the arguments in Section 1 of \cite{Iva87}, as detailed in Appendix~\ref{sec: Simple connectivity of the disk complexes of the standard Heegaard surfaces of the 3-sphere}.

We note that the reducing sphere complex for the genus-$2$ splitting has a slightly different definition, since any two non-isotopic reducing spheres must intersect in this case.
In \cite{Cho08}, its combinatorial structure is well-understood; in fact, it is a $2$-dimensional complex that deformation retracts to a tree.
This paper does not address the genus-$2$ splitting.

The main result of this work is stated as follows.

\begin{theorem}\label{main_theorem}
For the genus-$3$ Heegaard splitting of the $3$-sphere, the reducing sphere complex $\mathcal{R}(\Sigma)$ is connected.
\end{theorem}

The following is a direct consequence of Theorem \ref{main_theorem} together with Zupan \cite{Zupan}.

\begin{corollary}\label{Powell_conjecture}
The Powell Conjecture is true in the case of genus $3$.
\end{corollary}

\section{Primitive disks and weak reducing disks}\label{sec:primitive_disks}

Throughout the paper, $\N(X)$ and $\partial X$  will denote a regular neighborhood and the boundary of $X$, respectively, for a subspace $X$ of a space, where the ambient space will always be clear from the context.

\vspace{1em}

Let $(V, W; \Sigma)$ be a genus-$g$ Heegaard splitting of $S^3$ with $g \ge 3$.
An essential disk $D$ in $V$ is called a {\em primitive disk} if there is an essential disk $E$ in $W$ such that $\partial D$ intersects $\partial E$ transversely in a single point.
Such a disk $E$ is called a {\em dual disk} of $D$.
The disk $E$ is also a primitive disk in $W$ with a dual disk $D$ in $V$.
We note that any primitive disk is necessarily non-separating.
We call the pair $(D, E)$ of a primitive disk $D$ with its dual disk $E$ simply a {\em dual pair}.
Two dual pairs $(D, E)$ and $(D', E')$ are said to be {\em separated} if $D \cup E$ and  $D' \cup E'$ are disjoint.
Two dual pairs $(D, E)$ and $(D', E')$ are said to be {\em p-connected} if there exists a sequence of dual pairs $(D, E) = (D_1, E_1), (D_2, E_2), \ldots, (D_n, E_n) = (D', E')$ such that $(D_i, E_i)$ and $(D_{i+1}, E_{i+1})$ are separated for each $i \in \{ 1, 2, \ldots, n-1 \}$.
A reducing sphere $P$ is said to be {\em associated} with a dual pair $(D, E)$ if $P = \partial \N(D \cup E)$.

\begin{lemma}\label{lem0.1}
Let $(D, E)$ and $(D', E')$ be dual pairs.
Let $P$ and $P'$ be reducing spheres disjoint from $D \cup E$ and $D' \cup E'$, respectively.
Then $(D, E)$ and $(D', E')$ are p-connected if and only if the vertices of $\mathcal{R}(\Sigma)$ represented by $P$ and $P'$ are in the same component of $\mathcal{R}(\Sigma)$.
\end{lemma}

\begin{proof}
Suppose that the two dual pairs $(D, E)$ and $(D', E')$ are $p$-connected by a sequence of dual pairs $(D, E) = (D_1, E_1), (D_2, E_2), \ldots, (D_n, E_n) = (D', E')$.
Then we can choose a reducing sphere $P_i$ associated with the dual pair $(D_i, E_i)$ such that $P_i$ is disjoint from $P_{i+1}$ for each $i \in \{ 1, 2, \ldots, n-1 \}$.
Since $P$ is disjoint from $D \cup E$, we see that $P$ is disjoint from or isotopic to $P_1$, and similarly $P'$ is disjoint from or isotopic to $P_n$.
Thus we have a path in $\mathcal{R}(\Sigma)$ joining the vertices represented by $P$ and $P'$.

\vspace{0.2cm}

Conversely, suppose that we have a sequence $P = P_1, P_2, \ldots, P_n = P'$ of reducing spheres such that $P_i$ is disjoint from $P_{i+1}$ for each $i \in \{ 1, 2, \ldots, n-1 \}$.
The reducing sphere $P = P_1$ separates $S^3 = V \cup_{\Sigma} W$ into two $3$-balls $B_1 = V_1 \cup_{\Sigma_1} W_1$ and $B'_1 = V'_1 \cup_{\Sigma'_1} W'_1$, where $V_1 = V \cap B_1$, $W_1 = W \cap B_1$, $\Sigma_1 = \Sigma \cap B_1$, and $V'_1 = V \cap B'_1$, $W'_1 = W \cap B'_1$, $\Sigma'_1 = \Sigma \cap B'_1$.
Without loss of generality, assume that $D = D_1$ and $E = E_1$ are contained in $B_1 = V_1 \cup_{\Sigma_1} W_1$.

\vspace{0.2cm}

\noindent Case $1$. $P_2$ is contained in $V_1 \cup_{\Sigma_1} W_1$.

We can find a dual pair $(D_2, E_2)$ such that $D_2 \cup E_2$ is contained in $V'_1 \cup_{\Sigma'_1} W'_1$.
Then $D_2 \cup E_2$ is disjoint from $P_2$, and $(D_1, E_1)$ and $(D_2, E_2)$ are separated.

\vspace{0.2cm}

\noindent Case $2$. $P_2$ is contained in $V'_1 \cup_{\Sigma'_1} W'_1$.

As in the case of $P_1$, the reducing sphere $P_2$ also separates $V \cup_{\Sigma} W$ into two $3$-balls, say $V_2 \cup_{\Sigma_2} W_2$ and $V'_2 \cup_{\Sigma'_2} W'_2$, and without loss of generality, assume that $P_1$ is contained in $V_2 \cup_{\Sigma_2} W_2$.
We can take a dual pair $(D_2, E_2)$ such that $D_2 \cup E_2$ is contained in $V'_2 \cup_{\Sigma'_2} W'_2$, and is disjoint from $P_2$.
The dual pairs $(D_1, E_1)$ and $(D_2, E_2)$ are then separated.

\vspace{0.2cm}

Inductively, we take a dual pair $(D_i, E_i)$ such that $D_i \cup E_i$ is disjoint from $P_i$, and $(D_{i-1}, E_{i-1})$ and $(D_i, E_i)$ are separated for each $i \in \{ 2, \ldots, n \}$.
As in the cases of $P_1$ and $P_2$, the reducing sphere $P_n = P'$ separates $V \cup_{\Sigma} W$ into two $3$-balls, say $V_n \cup_{\Sigma_n} W_n$ and $V'_n \cup_{\Sigma'_n} W'_n$.
If $D_n \cup E_n$ and $D' \cup E'$ are in the same side of $P_n$, say in $V_n \cup_{\Sigma_n} W_n$, then take a dual pair $(D_{n+1}, E_{n+1})$ such that $D_{n+1} \cup E_{n+1}$ lies in $V'_n \cup_{\Sigma'_n} W'_n$.
Then $(D_n, E_n)$ and $(D', E')$ are $p$-connected via $(D_{n+1}, E_{n+1})$.
If $D_n \cup E_n$ and $D' \cup E'$ are in the opposite sides of $P_n$, then $(D_n, E_n)$ and $(D', E')$ are already separated.
We conclude that $(D, E)$ and $(D', E')$ are $p$-connected.
\end{proof}

\begin{lemma}\label{lem0.2}
Let $E$ be a common dual disk of two primitive disks $D$ and $D'$.
Then the dual pairs $(D, E)$ and $(D', E)$ are p-connected.
\end{lemma}

\begin{proof}
First suppose that $D$ is disjoint from $D'$.
Consider the reducing spheres $P$ and $P'$ associated with $(D, E)$ and $(D', E)$, respectively.
It is shown by Zupan that if two reducing spheres intersect at most in six points in $\Sigma$ then the vertices represented by them are in the same component of $\mathcal{R}(\Sigma)$ \cite[Theorem 1.3]{Zupan}.
Since $| (P \cap \Sigma) \cap (P' \cap \Sigma) | = 4$, the vertices represented by $P$ and $P'$ are in the same component of $\mathcal{R}(\Sigma)$.
Then by Lemma \ref{lem0.1}, $(D, E)$ and $(D', E)$ are $p$-connected.

Next suppose that $D$ intersects $D'$.
Let $p$ and $p'$ be the unique intersection points of $D$ and $E$, and $D'$ and $E$, respectively.
Let $\Delta$ be a subdisk of $D'$ cut off by an outermost arc $\delta$ of $D \cap D'$ in $D'$ such that $\Delta$ does not contain $p'$.
The arc $\delta$ cuts $D$ into two disks $F_1$ and $F_2$.
Let $F_1$ be the one that contains the point $p$, and let $D_2 = F_1 \cup \Delta$.
Since the disk $D_2$ intersects $E$ only in the point $p$, $(D_2, E)$ is a dual pair.
We observe that $| D_2 \cap D' | < | D \cap D' |$ after a slight isotopy since at least the arc $\delta$ no longer counts.

Inductively, we obtain a sequence of primitive disks $D = D_1, D_2, \ldots, D_n = D'$ in $V$ with the common dual disk $E$ such that $D_i$ is disjoint from $D_{i+1}$ for each $i \in \{ 1, 2, \ldots, n-1 \}$.
To show that $(D, E)$ and $(D', E)$ are $p$-connected, it is enough to show that $(D_i, E)$ and $(D_{i+1}, E)$ are $p$-connected for each $i$, but it follows from the same argument in the first paragraph.
\end{proof}

An essential disk $D$ in $V$ is called a {\em weak reducing disk} if there is an essential disk $E$ in $W$ disjoint from $D$.
The disk $E$ is also a weak reducing disk in $W$.
We call the pair of $D$ and $E$ a {\em weak reducing pair} and denote by $D - E$.
Of course $E - D$ is also a weak reducing pair.
Obviously, any primitive disk is a weak reducing disk.
That is, given any primitive disk $D$ in $V$ with a dual disk $E$ in $W$, the disk $\partial (\N (D \cup E)) \cap W$ in $W$ is disjoint from $D$.
A sequence of weak reducing disks $\Delta_1, \Delta_2, \ldots, \Delta_n$, denoted by $\Delta_1 - \Delta_2 - \cdots - \Delta_n$ is called a {\em weak reducing sequence} for $\Sigma$ (connecting $\Delta_1$ to $\Delta_n$) if the pair of $\Delta_i$ and $\Delta_{i+1}$ is a weak reducing pair $\Delta_i - \Delta_{i+1}$ for each $i \in \{ 1, 2, \ldots, n-1 \}$.
A weak reducing sequence is called a {\em primitive weak reducing sequence} if every disk in the sequence is a primitive disk.

\begin{lemma}\label{lem0.7}
Let $D_1 - E_2 - D_3 - \cdots - E_{2n} - D_{2n+1}$ be a primitive weak reducing sequence for $\Sigma$.
Then there exist dual disks $\Delta_1$ and $\Delta_{2n+1}$ of $D_1$ and $D_{2n+1}$, respectively, such that the dual pairs $(D_1, \Delta_1)$ and $(D_{2n+1}, \Delta_{2n+1})$ are p-connected.
\end{lemma}

\begin{proof}
We may assume that $D_1$ is contained in $V$.
Consider the first two primitive disks $D_1$ and $E_2$.
We show that there exists a dual disk $\Delta_1$ of $D_1$ disjoint from $E_2$ as follows.
Let $\Delta_1$ be a dual disk of $D_1$ that intersects $E_2$ minimally.
For contradiction, suppose $| \Delta_1 \cap E_2 | > 0$.
Choose any outermost subdisk of $E_2$ cut off by $\Delta_1 \cap E_2$.
Since $E_2$ is disjoint from $D_1$, one of the two disks, say $\Delta'_1$, obtained by surgery on $\Delta_1$ along the outermost subdisk would be again a dual disk of $D_1$, but we have $| \Delta'_1 \cap E_2 | < | \Delta_1 \cap E_2 |$, a contradiction.

Next, we show that there exists a dual disk $\Delta_2$ of $E_2$ disjoint from $D_1 \cup \Delta_1$.
The reducing sphere $P$ associated with the dual pair $(D_1, \Delta_1)$ separates $D_1$ and $E_2$.
Choose a dual disk $\Delta_2$ of $E_2$ that intersects the disk $P \cap V$ minimally.
It is enough to show that $\Delta_2$ is disjoint from $P \cap V$.
For contradiction, suppose $| \Delta_2 \cap (P \cap V) | > 0$.
Choose any outermost subdisk of $P \cap V$ cut off by $\Delta_2 \cap (P \cap V)$.
Since $P \cap V$ is disjoint from $E_2$, one of the two disks, say $\Delta'_2$, obtained by surgery on $\Delta_2$ along the outermost subdisk would be again a dual disk of $E_2$, but we have $| \Delta'_2 \cap (P \cap V) | < | \Delta_2 \cap (P \cap V) |$, a contradiction.

Now we have dual disks $\Delta_1$ and $\Delta_2$ of $D_1$ and $E_2$, respectively, such that the dual pairs $(D_1, \Delta_1)$ and $(\Delta_2, E_2)$ are separated.
In the same way, we have dual disks $\Delta'_2$ and $\Delta_3$ of $E_2$ and $D_3$, respectively, such that $(\Delta'_2, E_2)$ and $(D_3, \Delta_3)$ are separated.
The disk $E_2$ is a common dual disk of $\Delta_2$ and $\Delta'_2$.
By Lemma \ref{lem0.2}, $(\Delta_2, E_2)$ and $(\Delta'_2, E_2)$ are $p$-connected, and hence $(D_1, \Delta_1)$ and $(D_3, \Delta_3)$ are $p$-connected.

\begin{figure}[htbp]
\begin{center}
\includegraphics[width=14cm,clip]{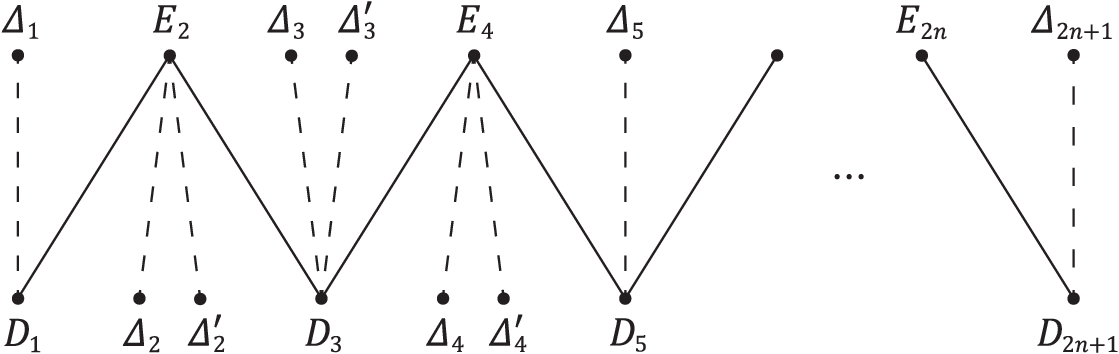}
\captionof{figure}{Finding dual pairs that are $p$-connected.}
\label{Fig2}
\end{center}
\end{figure}

Repeating the same argument, we finally obtain a dual disk $\Delta_{2n+1}$ of $D_{2n+1}$ such that $(D_1, \Delta_1)$ and $(D_{2n+1}, \Delta_{2n+1})$ are $p$-connected. See Figure \ref{Fig2}.
\end{proof}

In the special case that the genus of the splitting $(V, W; \Sigma)$ is exactly three, we can refine a short weak reducing sequence to a new one consisting only of non-separating disks in the middle.

\begin{lemma}\label{lem0.3}
Let $(V, W; \Sigma)$ be a genus-$3$ Heegaard splitting of $S^3$.
Suppose that $D - E - D'$ is a weak reducing sequence for $\Sigma$ such that $E$ is a separating disk in $W$.
Then there is a new weak reducing sequence $D - E' - D'$ or $D - E' - D'' - E'' - D'$ for $\Sigma$ such that $E'$, $D''$, $E''$ are all non-separating.
\end{lemma}

\begin{proof}
The disk $E$ cuts $W$ into two handlebodies $W_1$ and $W_2$.
Also the circle $\partial E$ cuts $\Sigma$ into two surfaces $\Sigma_1$ and $\Sigma_2$, where $\Sigma_i = \Sigma \cap W_i$ ($i=1,2$).

\vspace{0.5em}

\noindent Case $1$. Both of the loops $\partial D$ and $\partial D'$ lie in the same component of $\Sigma - \partial E$.

In this case, $\partial D \cup \partial D' $ lies in one, say $W_1$, of the two handlebodies $W_1$ and $W_2$.
Replacing $E$ with an essential non-separating disk $E'$ in $W_2$ such that $E'$ is disjoint from the scar of $E$ in $\partial W_2$, we obtain the desired weak reducing sequence $D - E' - D'$ for $\Sigma$.

\vspace{0.5em}

\noindent Case $2$. The loops $\partial D$ and $\partial D'$ lie in different components of $\Sigma - \partial E$.

If at least one of $\partial D$ and $\partial D'$, say $\partial D$, is isotopic to $\partial E$ in $\Sigma$, then we can move $\partial D$ by an isotopy so that the two loops $\partial D$ and $\partial D'$ lie in the same component of $\Sigma - \partial E$.
Then, we can apply the argument in Case $1$ to obtain the desired weak reducing sequence $D - E' - D'$ for $\Sigma$.

Suppose that neither $\partial D$ nor $\partial D'$ is isotopic to $\partial E$ in $\Sigma$.
Without loss of generality, we can assume that $W_1$ is a solid torus (and thus, $W_2$ is a genus-$2$ handlebody), $\partial D$ lies in $\Sigma_1$, and $\partial D'$ lies in $\Sigma_2$.
Thus the exterior of $W_1$ is also a solid torus, and $D$ is a meridian disk of it.
That is, $D$ is a primitive disk in $V$, and we can choose a dual disk $E''$ of $D$, a meridian disk of the solid torus $W_1$ so that $E''$ is disjoint from the scar of $E$ in $\partial W_1$.
Furthermore, we can find a reducing sphere $P$ associated with the dual pair $(D, E'')$ such that $P \cap W = E$.
The reducing sphere $P$ also cuts off a genus-$2$ handlebody $V_2$ from $V$ such that $V_2 \cap W_2 = \Sigma_2$.
Choosing non-separating essential disks $E'$ in $W_2$ and $D''$ in $V_2$ so that $E'$ is disjoint from $D''$ and $E' \cup D''$ is disjoint from $P$, we obtain the desired weak reducing sequence $D - E' - D'' - E'' - D'$ for $\Sigma$ (see Figure \ref{Fig3}).
\end{proof}

\begin{figure}[htbp]
\begin{center}
\includegraphics[width=6cm,clip]{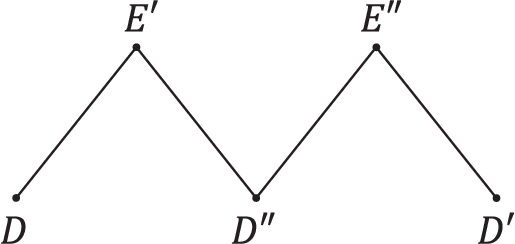}
\captionof{figure}{A weak reducing sequence with $E', D'', E''$ non-separating.}
\label{Fig3}
\end{center}
\end{figure}

A symmetric argument works for a weak reducing sequence $E - D - E'$, where $D$ is a separating disk in $V$.
Given any weak reducing sequence connecting $\Delta$ to $\Delta'$, applying Lemma \ref{lem0.3} repeatedly, we obtain a new weak reducing sequence connecting $\Delta$ to $\Delta'$ in which the intermediate disks are all non-separating.

\section{Topologically minimal surfaces}\label{sec:topologically_minimal_surface}

Bachman defined a critical surface \cite{Bachman1} and extended it to the notion of a topologically minimal surface \cite{Bachman3}.
Let $F$ be a surface in a $3$-manifold $M$, separating $M$ into two $3$-manifolds $V$ and $W$.
The {\em disk complex} $\mathcal{D}(F)$ is a simplicial complex defined as follows.
The vertices of $\mathcal{D}(F)$ are the isotopy classes of compressing disks for $F$ in $V$ and $W$, and $k+1$ distinct vertices span a $k$-simplex if they admit pairwise disjoint representatives.
The surface $F$ is then called a {\em critical surface} if the compressing disks for $F$ in $V$ and $W$ can be partitioned as $\mathcal C_1 \sqcup \mathcal C_2$ and $\mathcal C'_1 \sqcup \mathcal C'_2$, respectively, satisfying the following properties.

\begin{enumerate}
\item There exist compressing disks $D_i \in \mathcal C_i$ and $E_i \in \mathcal C'_i$ ($i = 1, 2$) such that $D_i \cap E_i = \emptyset$.
\item For any compressing disks $D_i \in \mathcal C_i$ and $E_{3-i} \in \mathcal C'_{3-i}$ ($i = 1, 2$), $D_i \cap E_{3-i} \ne \emptyset$.
\end{enumerate}

It is known that $F$ is critical if and only if $\pi_1(\mathcal{D}(F)) \ne 1$.
Also the following lemma is shown in \cite{Bachman2}.

\begin{lemma}\cite[Lemma 8.5]{Bachman2}\label{lem0.4}
If there exist two weak reducing disks $\Delta$ and $\Delta'$ for $F$ such that there is no weak reducing sequence connecting $\Delta$ to $\Delta'$, then $F$ is critical.
\end{lemma}

A surface $F$ is called a {\em topologically minimal surface} if $\mathcal{D}(F) = \emptyset$ or $\pi_i(\mathcal{D}(F)) \ne 1$ for some $i$.
The {\em topological index} of a topologically minimal surface is $0$ if $\mathcal{D}(F) = \emptyset$ and the smallest $i$ such that $\pi_{i-1}(\mathcal{D}(F)) \ne 1$ otherwise.
A critical surface is an index $2$ topologically minimal surface.

As mentioned in Introduction, it has been claimed in \cite{Appel} and in \cite{Campisi-Torres} that a genus-$g$ Heegaard surface $\Sigma$ in $S^3$ with $g \geq 2$ is topologically minimal with index $2g - 1$, which implies that $\mathcal{D}(\Sigma)$ is simply connected, and hence $\Sigma$ is not a critical surface.
In Appendix \ref{sec: Simple connectivity of the disk complexes of the standard Heegaard surfaces of the 3-sphere}, we provide an independent proof of the claim that $\mathcal{D}(\Sigma)$ is simply connected (see Theorem \ref{thm:simply_connected}).
Now we are ready to present a key lemma for the main theorem.
Still we assume that $(V, W; \Sigma)$ is a genus-$g$ Heegaard splitting of $S^3$ with $g \ge 3$.

\begin{lemma}\label{key_lemma}
Let $D$ and $D'$ be any non-separating weak reducing disks in $V$.
Then there exists a weak reducing sequence
\[ D = D_1  - E_2  - D_3 -  \cdots  - E_{2n}  - D_{2n+1} = D' \]
connecting $D$ to $D'$.
If the genus $g$ is three, we can additionally require that $D_i$  and $E_j$ are all non-separating for each $i \in \{ 3, 5, \ldots, 2n-1 \}$ and $j \in \{ 2, 4, \ldots, 2n \}$.
Furthermore, we may assume that $\partial D_i$ and $\partial D_{i+2}$ are not isotopic to each other in $\Sigma$ for $i \in \{ 1, 3, \ldots, 2n-1 \}$, and $\partial E_j$ and $\partial E_{j+2}$ are not isotopic to each other in $\Sigma$ for $j \in \{ 2, 4, \ldots, 2n-2 \}$.
\end{lemma}

\begin{proof}
By Theorem \ref{thm:simply_connected}, the Heegaard surface $\Sigma$ is not a critical surface, and hence by Lemma \ref{lem0.4}, there exists a weak reducing sequence connecting $D$ to $D'$.
In the particular case where the genus $g$ is three, this weak reducing sequence can be modified to be a weak reducing sequence
\[ D = D_1 - E_2 - D_3 -  \cdots - E_{2n} - D_{2n+1} = D' \]
connecting $D$ to $D'$ consisting only of non-separating disks by Lemma~\ref{lem0.3}.

If there exists some $i \in \{ 1, 3, \ldots, 2n-1 \}$ such that $\partial D_i$ and $\partial D_{i+2}$ are isotopic in $\Sigma$, then we can remove $E_{i+1}$ and $D_{i+2}$ from the sequence to obtain a shorter weak reducing sequence from $D$ to $D'$.
Similarly, if there exists some $j \in \{ 2, 4, \ldots, 2n-2 \}$ such that $\partial E_j$ and $\partial E_{j+2}$ are isotopic in $\Sigma$, then we can remove $D_{j+1}$ and $E_{j+2}$ from the sequence.
By repeating this process finitely many times, we finally obtain the desired weak reducing sequence.
\end{proof}

\section{Proof of the main theorem}\label{sec:proof_of_main_theorem}

Throughout the section, let $(V, W; \Sigma)$ be the genus-$3$ Heegaard splitting of $S^3$.

\begin{lemma}\label{lem1.1}
Suppose that $D - E - D'$ is a weak reducing sequence such that $D$, $E$, and $D'$ are all non-separating.
If $\partial D$ and $\partial D'$ are not isotopic to each other in $\Sigma$, then $E$ is a primitive disk.
\end{lemma}

\begin{proof}
We may assume that $D$ and $D'$ are contained in $V$.

\vspace{0.5em}

\noindent Case $1$.
The two disks $D$ and $D'$ are disjoint from each other.

If we compress $\Sigma$ along $E$, we obtain a surface $\Sigma_0$ containing two scars of $E$.
When we further compress $\Sigma_0$ along $D \cup D'$, if the two scars of $E$ lie in different components after the compression, we obtain a non-separating closed surface in $S^3$, which is a contradiction.
Thus, the two scars of $E$ lie in the same component, say $\Sigma_1$, of the surface obtained by compressing $\Sigma_0$ along $D \cup D'$.
We can see that $\Sigma_1$ also contains either
\begin{itemize}
\item the two scars of $D$ and the two scars of $D'$, or
\item one scar of $D$ and one scar of $D'$,
\end{itemize}
according as $D \cup D'$ is non-separating or separating in $V$.
Since the genus of $\Sigma$ is three, it is easily seen that the surface $\Sigma_1$ is the $2$-sphere.
Let $\gamma$ be a circle in $\Sigma_1$ separating the scars of $E$ from the scars of $D$ and $D'$.
The circle $\gamma$ can be regarded as a loop in $\Sigma$ and bounds essential disks in both $V$ and $W$.
Thus $\gamma$ is the intersection of $\Sigma$ with a reducing sphere $P$.
The reducing sphere $P$ separates $E$ from $D \cup D'$.
The disk $E$ is contained in the side of $P$ that is the genus-$1$ summand, and hence $E$ is a primitive disk.

\vspace{0.5em}

\noindent Case $2$.
The two disks $D$ and $D'$ intersect each other.

Suppose that $D$ intersects $D'$ minimally.
We perform surgery on the disk $D$ along an outermost subdisk of $D'$ cut off by $D \cap D'$.
At least one of the two disks obtained by the surgery is then non-separating, which we denote by $G$.
Obviously, $G$ is disjoint from $D$ and from $E$, and hence we have a new weak reducing sequence $D - E - G$, consisting of non-separating disks.
We observe that $D$ and $G$ are not isotopic to each other, since we assumed $D$ intersects $D'$ minimally.
Thus, by the same arguments in Case $1$, we see that $E$ is a primitive disk.
\end{proof}

\begin{proof}[Proof of Theorem \ref{main_theorem}]
Let $P$ and $P'$ be any two reducing spheres.
We choose dual pairs $(D, E)$ and $(D', E')$ such that $P$ and $P'$ are disjoint from $D \cup E$ and $D' \cup E'$, respectively.
To show that $\mathcal{R}(\Sigma)$ is connected, that is, to find a path joining the vertices represented by $P$ and $P'$ in $\mathcal{R}(\Sigma)$, it is enough to show that $(D, E)$ and $(D', E')$ are $p$-connected, by Lemma \ref{lem0.1}.

By Lemma \ref{key_lemma} and Lemma \ref{lem1.1}, we have a primitive weak reducing sequence connecting $D$ and $D'$.
By Lemma \ref{lem0.7}, there exist dual disks $\Delta$ and $\Delta'$ of $D$ and $D'$, respectively, such that $(D, \Delta)$ and $(D', \Delta')$ are $p$-connected.
By Lemma \ref{lem0.2}, the dual pairs $(D, E)$ and $(D, \Delta)$ are $p$-connected, and also $(D', \Delta')$ and $(D', E')$ are $p$-connected.
Hence $(D, E)$ and $(D', E')$ are $p$-connected.
\end{proof}

We remark that the proofs of Lemma \ref{lem1.1} and Theorem \ref{main_theorem} heavily depend on the fact that the genus of the splitting is three.
But we still expect that the following question has an affirmative answer for the splittings of higher genera.

\begin{question}
Given a genus-$g$ Heegaard splitting $(V, W; \Sigma)$ of the $3$-sphere with $g \geq 3$, and given any two primitive disks $D$ and $D'$ in $V$, can we always find a primitive weak reducing sequence connecting $D$ to $D'$?
\end{question}

\medskip

\appendix
\section{Simple connectivity of the disk complexes of the standard Heegaard surfaces of the 3-sphere}
\label{sec: Simple connectivity of the disk complexes of the standard Heegaard surfaces of the 3-sphere}

Let $(V, W; \Sigma)$ be a genus-$g$ Heegaard splitting of the 3-sphere $S^3$ for $g \geq 2$, and let $\mathcal{D}(\Sigma)$ denote the disk complex associated with the Heegaard surface $\Sigma$.
The full subcomplexes of $\mathcal{D}(\Sigma)$, spanned by the vertices corresponding to compressing disks in $V$ and $W$, respectively, are both contractible.
Since these subcomplexes are joined by edges represented by weak reducing pairs, it follows that $\mathcal{D}(\Sigma)$ is connected.
Our goal is to demonstrate that $\mathcal{D}(\Sigma)$ is simply connected.

Throughout this appendix, all maps and surfaces are assumed to be smooth.
We begin by summarizing some results introduced by Ivanov in \cite{Iva02}.
To prove the connectedness of the curve complex for a surface, Ivanov used a result on $1$-parameter families of functions on a surface due to Cerf \cite{Cerf_1}, \cite{Cerf_2}, specifically Lemma 3.2.A in \cite{Iva02}.
He then extended this to a version involving 2-parameter families of functions, which was also proven in \cite{Cerf_1}.
We restate it here as follows.

\begin{lemma}\label{lemma:2-parameter_family}
Let $\{ f_{\beta} : \Sigma \rightarrow \mathbb{R} \}_{\beta \in B}$ be any continuous family of functions, where the parameter space $B$ is the unit disk.
Then $\{ f_{\beta} \}_{\beta \in B}$ can be approximated arbitrarily closely by a continuous family of functions $\{ g_{\beta} \}_{\beta \in B}$, where each $g_{\beta}$ belongs to one of the following three classes:
\begin{itemize}
    \item Morse functions with all critical values distinct;
    \item Morse functions with exactly two or three identical critical values, with all other critical values distinct from each other and from the identical values;
    \item functions where all critical values are distinct, except for exactly one non-Morse critical point. In appropriate local coordinates $(x, y)$ around this critical point, the function has the form $x^3 \pm y^2 + c$ or $\pm x^4 \pm y^2 + c$ for some constant $c \in \mathbb{R}$.
\end{itemize}
\end{lemma}

A function $f : \Sigma \rightarrow \mathbb{R}$ is said to be \emph{admissible}\footnote{In Ivanov \cite{Iva87}, this is translated as a \emph{non-degenerate} function. However, we avoid this term here as it is typically associated with another class of functions.}
if $f$ admits a level component that is a non-trivial circle containing no critical points.
The following lemma corresponds exactly to a claim made in the proof of Theorem 3.2.B in \cite{Iva02}.

\begin{lemma}\label{lemma:level_component}
Any function $f$ belonging to one of the three classes in Lemma \ref{lemma:2-parameter_family} is admissible.
\end{lemma}

The key property needed for the proof of Lemma \ref{lemma:level_component} is that any function in one of the three classes listed in Lemma \ref{lemma:2-parameter_family} has a critical point that is neither a local maximum nor a local minimum.

\medskip
In the arguments that follow, we fix a Morse function $p : S^3 \rightarrow \mathbb{R}$ with exactly two critical points $x_+$ and $x_-$, such that $\Sigma \cap \{x_+, x_-\} = \emptyset$.
Let $\mathrm{Diff}(S^3, x_+, x_-)$ denote the topological group of self-diffeomorphisms of $S^3$ that map $x_+$ to $x_+$ and $x_-$ to $x_-$, and let $\mathrm{Diff}_0(S^3, x_+, x_-)$ denote the connected component of this group that contains the identity map $\mathrm{id}_{S^3}$.
A function $f : \Sigma \rightarrow \mathbb{R}$ is called a \emph{height function} if there exists a diffeomorphism $\iota \in \mathrm{Diff}_0(S^3, x_+, x_-)$ such that $f = p \circ \iota |_{\Sigma}$, where $p$ is the projection defined above.

\begin{lemma}\label{lem: innermost}
If a height function $f : \Sigma \rightarrow \mathbb{R}$ is admissible, then there exists a compressing disk for $\Sigma$ whose boundary is a component of a level set of $f$, with no critical points of $f$ on its boundary.
\end{lemma}

\begin{proof}
Let $\iota \in \mathrm{Diff}_0(S^3, x_+, x_-)$ be a diffeomorphism such that $f = p \circ \iota |_{\Sigma}$.
Suppose that $f^{-1}(s)$ contains a component $C \subset \Sigma$ that is a non-trivial circle without any critical points, where $s \in \mathbb{R}$.
Then there exists a sufficiently small $\varepsilon > 0$ such that for any $s' \in (s - \varepsilon, s + \varepsilon)$, the preimage $f^{-1}(s')$ contains a non-trivial circle component $C' \subset \Sigma$ near $C$ that also has no critical points and is isotopic to $C$ in $\Sigma$.

By Sard's theorem, we can find a real number $s' \in (s - \varepsilon, s + \varepsilon)$ that is not a critical value.
Then $f^{-1}(s')$ consists of finitely many circle components lying in the level sphere $S := (p \circ \iota)^{-1}(s')$.
Choose a component $C_0$ of $f^{-1}(s')$ that is innermost in $S$ among those that are non-trivial in $\Sigma$.
Then, it can be shown easily that $C_0$ bounds a compressing disk $D_0$, noting that $D_0$ does not necessarily lie within $S$.
\end{proof}

Equip a product Riemannian metric on $S^3 \setminus \{ x_+, x_- \}$ by identifying $S^3 \setminus \{ x_+, x_- \}$ with the product space $S^2 \times \mathbb{R}$ so that each $S^2 \times \{ t \}$ corresponds to a level sphere of the Morse function $p: S^3 \to \mathbb{R}$.

Let $\{ \iota_\beta \in \mathrm{Diff}_0(S^3, x_+, x_-) \}_{\beta \in B}$ be a smooth family of self-diffeomorphisms of $S^3$.
Let $P$ denote the closed subset of $\Sigma \times B$ consisting of the pairs $(x, \beta)$ such that the tangent plane $T_{\iota_\beta(x)} \iota_\beta(\Sigma)$ at $\iota_\beta(x)$ of $\iota_\beta(\Sigma)$ is perpendicular to the tangent plane at $\iota_\beta(x)$ of the level sphere containing $\iota_\beta(x)$.

Consider the continuous family $\{ f_\beta := p \circ \iota_{\beta} |_{\Sigma} : \Sigma \to \mathbb{R} \}_{\beta \in B}$ of height functions.
Note that for $(x, \beta) \in P$, $x$ is not a critical point of the height function $f_\beta$.

Let $(x, \beta) \in (\Sigma \times B) \setminus P$.
For simplicity, we assume that $f_\beta(x) = 0$.
Let $U$ be a sufficiently small open neighborhood of $\iota_\beta(x)$ in $S^3$.
We can identify $U$ with an open neighborhood $U'$ of the origin $\boldsymbol{0}$ in $\mathbb{R}^3 = \{ (u, v, w) \mid u, v, w \in \mathbb{R} \}$ so that $p|_{U'}$ maps $(u, v, w)$ to $w$, and $U \cap \iota_{\beta}(\Sigma)$ is identified with the graph $\{ (u, v, \varphi(u, v)) \} \subset U'$ of a certain smooth function $\varphi$.
With this identification, we can treat $(u, v)$ as a local coordinate for $\Sigma$ centered at $x = (0, 0)$.
Since the graph of any smooth function $\psi$ defined near $(0, 0)$ can be obtained from that of $\varphi$ by an isotopy of the form $\{ (u, v, (1 - t)\varphi(u, v) + t\psi(u, v)) \}_{t \in [0, 1]}$, we see that the set of germs of height functions $(\Sigma, x) \to (\mathbb{R}, 0)$ is exactly the same as the set of germs of all smooth functions $(\Sigma, x) \to (\mathbb{R}, 0)$, which is identified with $(\mathbb{R}^2, (0,0)) \to (\mathbb{R}, 0)$ using the local coordinate $(u, v)$.
Furthermore, the existence of the above isotopy implies that for any small perturbation $\{ g_\beta \}$ near $(x, \beta_0)$ of the family $\{ f_\beta \}$, there exists a small perturbation $\{ \iota'_\beta \in \mathrm{Diff}_0(S^3, x_+, x_-) \}$ of $\{ \iota_\beta \in \mathrm{Diff}_0(S^3, x_+, x_-) \}$ such that $g_\beta = p \circ \iota'_{\beta} |_{\Sigma}$ for $\beta \in U_{\beta_0}$, where $U_{\beta_0}$ is a sufficiently small open neighborhood of $\beta_0$ in $B$.
The support of necessary perturbation of $\{ f_\beta \}_{\beta \in B}$ to obtain a family $\{ g_\beta \}_{\beta \in B}$ of smooth functions as in Lemma \ref{lemma:2-parameter_family} is performed within $(\Sigma \times B) \setminus P$.
Therefore, we have the following lemma, which is a version of Lemma \ref{lemma:2-parameter_family} for height functions.

\begin{lemma}\label{lemma:2-parameter_family of height functions}
Let $\{ f_{\beta} : \Sigma \rightarrow \mathbb{R} \}_{\beta \in B}$ be any continuous family of height functions, where the parameter space $B$ is the unit disk.
Then $\{ f_{\beta} \}_{\beta \in B}$ can be approximated arbitrarily closely by a continuous family of height functions $\{ g_{\beta} \}_{\beta \in B}$, where each $g_{\beta}$ belongs to one of the three classes listed in Lemma \ref{lemma:2-parameter_family}.
\end{lemma}

We now proceed to prove our main theorem.
For convenience, we do not distinguish between a simplicial complex and its underlying space.

\begin{theorem}\label{thm:simply_connected}
The disk complex $\mathcal{D}(\Sigma)$ is simply connected.
\end{theorem}

\begin{proof}
Let $h : S^1 \rightarrow \mathcal{D}(\Sigma)$ be any simplicial map, where $S^1$ is the unit circle endowed with a triangulation.
Let $\{ \alpha_0, \alpha_1, \ldots, \alpha_{k-1} \}$ be the collection of vertices of $S^1$ in this order.
We regard $S^1 = \{\exp(2 \pi \theta \sqrt{-1}) \mid \theta \in \mathbb{R}\} \subset \mathbb{C}$ and $\alpha_i := \exp \left( \frac{2 \pi i \sqrt{-1}}{k} \right) $.
It is enough to show that $h$ extends to a simplicial map $\overline{h} : B \rightarrow \mathcal{D}(\Sigma)$, where $B = \{ r \alpha \mid 0 \leq r \leq 1, ~ \alpha \in S^1 \}$ is the unit disk endowed with a triangulation such that the set of vertices in $\partial B = S^1$ coincides with the set $\{ \alpha_0, \alpha_1, \ldots, \alpha_{k-1} \}$.

For each $i \in \{ 0, 1, \ldots, k-1 \}$, we choose a compressing disk $D_i$ for $\Sigma$ that represents the vertex $h(\alpha_i)$ of $\mathcal{D}(\Sigma)$.
By deforming $h$ homotopically if necessary, we may assume that any two consecutive disks $D_i$ and $D_{i+1}$ (subscripts $\mathrm{mod} ~ k$) are not isotopic.
We can thus further assume that $D_i$ and $D_{i+1}$ (subscripts $\mathrm{mod} ~ k$) are disjoint.
Denote by $C_i$ the boundary circle of $D_i$.
For each $i$, choose a sufficiently small closed $3$-ball neighborhood $N_i$ of $D_i$ so that $N_i \cap N_{i+1} = \emptyset$.
Then define an isotopy $\{ (\eta_i)_t \in \mathrm{Diff}(S^3, x_+, x_-) \}_{t \in [0, 1]}$ so that
\begin{itemize}
    \item $(\eta_i)_0 := \mathrm{id}_{S^3}$;
    \item outside of $N_i$, $(\eta_i)_t$ is the identity for each $t \in [0, 1]$; and
    \item $(\eta_i)_1 (D_i)$ lies in a level set of $p$.
\end{itemize}

Set $f_{\alpha_i} := p \circ (\eta_i)_1 |_{\Sigma}$.
By construction, $C_i$ is a level circle of $f_{\alpha_i}$.
Note that for any $s, t \in [0, 1]$, we have $(\eta_i)_s \circ  (\eta_{i+1})_{t} = (\eta_{i+1})_{t} \circ  (\eta_i)_s$, since their supports $N_i$ and $N_{i+1}$ are disjoint.
Define an $S^1$-family $\{\iota_\alpha \in \mathrm{Diff}(S^3, x_+, x_-)\}_{\alpha \in S^1}$ of self-diffeomorphisms of $S^3$ by setting, for $\alpha = \exp \left( \frac{2 \pi \theta \sqrt{-1}}{k} \right)$ ($i \leq \theta \leq i+1$),
\[
\iota_\alpha = \begin{cases}
(\eta_{i+1})_{2(\theta - i)} \circ (\eta_i)_1 & (i \leq \theta \leq \frac{2i + 1}{2}), \\
(\eta_{i+1})_1 \circ (\eta_i)_{2 (i+1 - \theta)} & (\frac{2i+1}{2} \leq \theta \leq i+1).
\end{cases}
\]
Then we have $p \circ \iota_{\alpha_i} |_{\Sigma} = p \circ (\eta_i)_1 |_{\Sigma} = f_{\alpha_i}$.
Setting $f_{\alpha} := p \circ \iota_\alpha |_{\Sigma}$ ($\alpha \in S^1$), we obtain an $S^1$-family of height functions on $\Sigma$ such that
\begin{itemize}
    \item for any $\frac{i}{k} \leq \theta \leq \frac{2i+1}{2k},$ $C_i$ is a level circle of $f_{ \exp (2 \pi \theta \sqrt{-1}) }$; and
    \item for any $\frac{2i+1}{2k} \leq \theta \leq \frac{i+1}{k}$, $C_{i+1}$ is a level circle of $f_{\exp( 2 \pi \theta \sqrt{-1}) }$.
\end{itemize}

Now we want to construct a triangulation of the unit disk $B$ such that the triangulated unit circle $S^1$ is a subcomplex of $B$ under this triangulation.
Define a $B$-family $\{\iota_{r \alpha} \in \mathrm{Diff}(S^3, x_+, x_-)\}_{\alpha \in S^1}$ of self-diffeomorphisms of $S^3$ by setting, for $r \alpha = r \exp \left( \frac{2 \pi \theta \sqrt{-1}}{k} \right)$ ($i \leq \theta \leq i+1$),
\[
\iota_{r \alpha} = \begin{cases}
(\eta_{i+1})_{2r(\theta - i)} \circ (\eta_i)_r & (i \leq \theta \leq \frac{2i + 1}{2}), \\
(\eta_{i+1})_r \circ (\eta_i)_{2r (i+1 - \theta)} & (\frac{2i+1}{2} \leq \theta \leq i+1).
\end{cases}
\]
Setting $f_{r \alpha} = p \circ \iota_{r \alpha} |_{\Sigma}$, we have a continuous family of height functions $\{f_{\beta}\}_{\beta = r \alpha \in B}$.

By Lemma \ref{lemma:2-parameter_family of height functions}, the family $\{f_{\beta}\}_{\beta \in B}$ is approximated arbitrarily closely by a continuous family of height functions $\{g_{\beta}\}_{\beta \in B}$ such that each $g_\beta$ belongs to one of the three classes in Lemma \ref{lemma:2-parameter_family}.

For each $\beta \in B$, by Lemma \ref{lemma:level_component}, each $g_{\beta}$ is admissible and thus, by Lemma \ref{lem: innermost}, we can choose a level component $C_{\beta}$ such that $C_{\beta}$ bounds a compressing disk, say $D_{\beta}$, for $\Sigma$ and $C_{\beta}$ contains no critical points of $g_{\beta}$.
In particular, for each $\alpha = \exp \left( \frac{2 \pi \theta \sqrt{-1}}{k} \right) \in S^1$, $\frac{2i-1}{2} \leq \theta \leq \frac{2i+1}{2}$, we can choose such a level component $C_{\alpha}$ of $g_{\alpha}$ to be isotopic to $C_i$ since $g_{\alpha}$ is sufficiently close to $f_{\alpha}$.

If $\beta'$ is sufficiently close to $\beta$, say $\beta'$ belongs to some small open neighborhood $U_{\beta}$ of $\beta$, then the function $g_{\beta'}$ has a level component $C_{\beta'}$ that is isotopic to $C_{\beta}$ since $C_{\beta}$ contains no critical points of $g_{\beta}$.
The circle $C_{\beta'}$ bounds, of course, a compressing disk isotopic to $D_{\beta}$.
The collection $\{U_{\beta}\}_{\beta \in B}$ of such open neighborhoods of every point of $B$ forms an open cover of the disk $B$.
In particular, for the vertices $\{ \alpha_0, \alpha_1, \ldots, \alpha_{k-1} \}$ of $S^1 = \partial B$, we can choose an open neighborhood $U_{\alpha_i}$ that contains $\left\{ \exp \left( \frac{2 \pi \theta \sqrt{-1}}{k} \right) \in S^1 \mid \frac{2i-1}{2} \leq \theta \leq \frac{2i+1}{2} \right\}$, and so the finite subcollection $\{ U_{\alpha_0}, U_{\alpha_1}, \ldots, U_{\alpha_{k-1}} \}$ covers $S^1$.
Since $B$ is compact, we can choose a finite subcover $\{U_{\beta}\}_{\beta \in B'}$ of $\{U_{\beta}\}_{\beta \in B}$.
We can choose such a subcover so that $B' \cap S^1 = \{ \alpha_0, \alpha_1, \ldots, \alpha_{k-1} \}$, and furthermore so that the underlying space of the \emph{nerve} $X$ of this subcover is identified with the disk $B$.
That is, $X$ is a triangulation of $B$ whose vertex set is $B'$.
Clearly, the original triangulated $S^1$ is a subcomplex of the triangulated $B$.

Consider the situation when $U_{\beta} \cap U_{\beta'} \neq \emptyset$.
We have chosen the level components $C_{\beta}$ and $C_{\beta'}$ of $g_{\beta}$ and $g_{\beta'}$, bounding compressing disks $D_{\beta}$ and $D_{\beta'}$ respectively.
Choosing $\gamma \in U_{\beta} \cap U_{\beta'}$, the function $g_{\gamma}$ has level components $C_{\gamma}$ and $C_{\gamma'}$ which are isotopic to $C_{\beta}$ and $C_{\beta'}$ respectively.
Since $C_{\gamma}$ and $C_{\gamma'}$ are level components of the same function $g_{\gamma}$, they are either equal or disjoint from each other.
That implies the two compressing disks $D_{\beta}$ and $D_{\beta'}$ are equal or disjoint from each other, up to isotopy.
In other words, the vertices of $\mathcal{D}(\Sigma)$ represented by $D_{\beta}$ and $D_{\beta'}$ are either the same or joined by an edge.
Now, assigning each vertex $\beta \in B'$ of the triangulation of $B$ to the vertex of $\mathcal{D}(\Sigma)$ represented by the disk $D_{\beta}$ we have chosen, we eventually obtain the desired simplicial map $\overline{h}: B \rightarrow \mathcal{D}(\Sigma)$.
\end{proof}

\noindent
{\bf Acknowledgements}\, The authors extend their heartfelt thanks to Martin Scharlemann for his valuable discussions and insightful feedback. The second-named author wishes to express his deepest gratitude to Masaharu Ishikawa for his invaluable guidance and profound advice on understanding the works of Arnold and Ivanov in the field of singularity theory of real functions.

\bibliographystyle{amsplain}

\end{document}